\documentclass[12pt]{amsart}
\usepackage{mathpazo}
\usepackage{hyperref}
\usepackage{cleveref}
\usepackage{bm}
\usepackage{comment}
\usepackage[margin=1in]{geometry}
\usepackage{setspace}
\usepackage{enumitem}

\usepackage{graphicx}
\usepackage{tikz-cd}

\usepackage[cmtip, all]{xy}
\usepackage{array}

\usepackage{amssymb,amsmath,latexsym,amsthm}

\DeclareMathOperator{\Aut}{Aut}
\DeclareMathOperator{\Jac}{Jac}

\DeclareMathOperator{\Sets}{Sets}
\DeclareMathOperator{\val}{val}
\DeclareMathOperator{\Stab}{Stab}

\DeclareMathOperator{\ns}{ns}

\newcommand{\field}[1]{\mathbb{#1}}

\newcommand{\Z}{\field{Z}}

\newcommand{\R}{\field{R}}

\newcommand{\beq}{\begin{displaymath}}
\newcommand{\eeq}{\end{displaymath}}
\newcommand{\beqn}{\begin{equation}}
\newcommand{\eeqn}{\end{equation}}

\newcommand{\cyc}[1]{[#1]_2}
\newcommand{\cM}{\mathcal{M}}

\newcommand{\ED}{\mathcal{E}}

\newcommand{\wG}{\mathbf{G}}

\numberwithin{equation}{section}
\newtheorem{theorem}[equation]{Theorem}
\newtheorem{lemma}[equation]{Lemma}

\newtheorem{proposition}[equation]{Proposition}

\theoremstyle{definition}

\newtheorem*{theorem*}{Theorem}

\newcommand{\setleftmargin}[1]{
	\addtolength{\textwidth}{\oddsidemargin}
	\addtolength{\textwidth}{1in}
	\addtolength{\textwidth}{-#1}
	\setlength{\oddsidemargin}{-1in}
	\addtolength{\oddsidemargin}{#1}
	\setlength{\evensidemargin}{\oddsidemargin}
}

\newcommand{\setrightmargin}[1]{
	\setlength{\textwidth}{8.5in}
	\addtolength{\textwidth}{-\oddsidemargin}
	\addtolength{\textwidth}{-1in}
	\addtolength{\textwidth}{-#1}
}

\newcommand{\setopmargin}[1]{
	\addtolength{\textheight}{\topmargin}
	\addtolength{\textheight}{1in}
	\addtolength{\textheight}{\headheight}
	\addtolength{\textheight}{\headsep}
	\addtolength{\textheight}{-#1}
	\setlength{\topmargin}{-1in}
	\addtolength{\topmargin}{-\headheight}
	\addtolength{\topmargin}{-\headsep}
	\addtolength{\topmargin}{#1}
}

\newcommand{\setbottommargin}[1]{
	\setlength{\textheight}{11in}
	\addtolength{\textheight}{-\topmargin}
	\addtolength{\textheight}{-1in}
	\addtolength{\textheight}{-\footskip}
	\addtolength{\textheight}{-#1}
}

\newcommand{\setallmargins}[1]{
	\setbottommargin{#1}
	\setopmargin{#1}
	\setrightmargin{#1}
	\setleftmargin{#1}
}

\setallmargins{1in}

\title{Tropical curves of hyperelliptic type}

\author{Daniel Corey}
\address{UW Department of Mathematics, Van Vleck Hall, 480 Lincoln Dr., Madison, WI  53706}
\email{dcorey@math.wisc.edu}

\begin{document}

\begin{abstract}
We introduce the notion of tropical curves of hyperelliptic type.  These are tropical curves whose Jacobian is isomorphic to that of a hyperelliptic tropical curve, as polarized tropical abelian varieties. We show that this property depends only on the underlying graph of a tropical curve and is preserved when passing to genus $\geq 2$ connected minors. The main result is an excluded minors characterization of tropical curves of hyperelliptic type.
\end{abstract}

\maketitle

\section{Introduction}

The classical Torelli theorem asserts that two algebraic curves are isomorphic if and only if their Jacobians are isomorphic as polarized abelian varieties.  It is well known that the tropical analogue of this theorem is not true, see \cite[Section~6.4]{MikhalkinZharkov}. 
For example, varying the lengths of a separating pair of edges while preserving their sum produce tropical curves with isomorphic Jacobians. This phenomenon is presented in~\Cref{fig:sameJac}. As a consequence, it is possible for a non-hyperelliptic tropical curve to have a Jacobian isomorphic to that of a hyperelliptic tropical curve. We say that such a tropical curve is of \textit{hyperelliptic type}. A natural problem is to classify these objects. 
 
\begin{figure}[htb]
	\begin{center}
		\includegraphics[height=20mm]{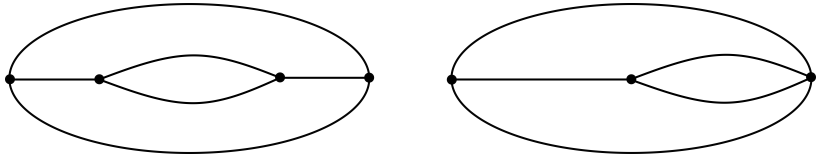}
	\end{center}
	\caption{Two tropical curves with isomorphic Jacobians.  \label{fig:sameJac}}
\end{figure}

The property of being hyperelliptic type has some interesting characteristics.  It is  independent of the edge lengths and preserved when passing to connected genus $\geq 2$ minors.   Our main theorem is a forbidden minors classification of these tropical curves. 
\begin{theorem*}
	A tropical curve $\Gamma$ is of hyperelliptic type if and only if the underlying graph of $\Gamma$ does not have  $K_4$ or $L_3$ as a minor. 
\end{theorem*}
\noindent Here, $K_4$ is the complete graph on $4$ vertices and $L_3$ is the ``loop of 3 loops,'' both graphs are displayed in \Cref{fig:k4l3}. We prove a slightly stronger result, see~\Cref{thm:HTK4L3}.

\begin{figure}[htb]
	\begin{center}
		\includegraphics[height=25mm]{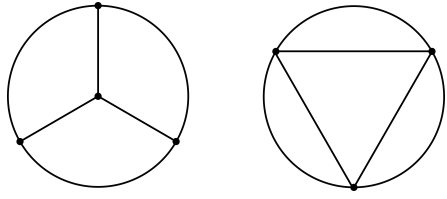}
	\end{center}
	
	\caption{The graphs $K_4$ (left) and $L_3$ (right). \label{fig:k4l3}}
\end{figure}

In~\Cref{sec:Preliminaries}, we review topics in tropical curves that we will need in this paper, including a discussion of C1-sets and 3-edge connectivization. A key tool used throughout the paper is the tropical Torelli theorem of Caporaso and Viviani \cite{CaporasoViviani}.  This will be used in \Cref{sec:HyperellipticTypeAndItsProperties} to give a description of hyperelliptic type in terms of 3-edge connectivizations. We will also show that hyperelliptic type depends only on the underlying graph, and that it is a minor closed property. Graphs that have no $K_4$ or $L_3$ minor admit a particular type of nested ear decompositions. These will be introduced in \Cref{sec:NestedEarDecompositions}, which will then be used to prove the main theorem. 

\smallskip

\noindent \textbf{Acknowledgments.} The author would like to thank Jordan Ellenberg, David Jensen, and Wanlin Li for helpful conversations, and Dmitry Zakharov for comments on an earlier draft. This research is partially supported by NSF RTG Award DMS--1502553. 

\section{Preliminaries}
\label{sec:Preliminaries}
\subsection{Tropical curves}
A \textit{weighted graph} $\mathbf{G} = (G,w)$ is a finite connected graph $G$ (possibly with loops or multiple edges) together with a function $w:V(G) \to \Z_{\geq 0}$ recording the weights of the vertices. Each edge $e\in E(G)$ is viewed as a pair of distinct half-edges. We write $e=vw$ to indicate that the endpoints of $e$ are the vertices $v$ and $w$. The \textit{valence} of a vertex $v$ is the number of half-edges incident to $v$, written $\val(v)$. In particular, a loop counts for two incidences. A vertex $v$ is  \textit{stable} if 
\begin{equation*}
2\,w(v) - 2 + \val(v) > 0,
\end{equation*}
and a weighted graph is \textit{stable} if each vertex is stable.  The \textit{genus} of $\wG$ is 
\begin{equation*}
g(\wG) = b_1(G) + |w|
\end{equation*}
where $b_1$ is the first Betti number of $G$ and $|w|$ is the sum of the weights. For an edge $e$ of $\wG$, the \textit{contraction} of $\wG$ by $e$ is the weighted graph $\wG/e$ obtained by contracting $e$ while changing the weight function in the following way.  If $e$ is a loop edge incident to $v$, then the weight of $v$ increases by 1. If $e$ is an edge between distinct vertices $v_1$ and $v_2$, then the weight of the new vertex is $w(v_1) + w(v_2)$.  Note that this preserves the genus and stability of $\wG$. If $\wG'$ is obtained from $\wG$ by a sequence of contractions, then $\wG'$ is a \textit{specialization} of $\wG$.    

A \textit{tropical curve} $\Gamma$ is a weighted graph $\wG$ together with a function $\ell: E(G) \to \R_{>0}$ recording the length of each edge. 
Every tropical curve of genus $g\geq 2$ is tropically equivalent to a unique tropical curve whose underlying weighted graph is stable (see \cite[Section~2]{Caporaso}). We refer to this as the \textit{stable model} for $\Gamma$. 

Let $\Gamma$ be a genus $g \geq 2$ tropical curve. The \textit{Jacobian} of $\Gamma$ is the real $g$ dimensional torus 
\begin{equation*}
\Jac(\Gamma) = (H_1(\Gamma,\R) \oplus \R^{|w|})/(H_1(\Gamma,\Z) \oplus \Z^{|w|}) 
\end{equation*}
together with the semi-positive quadratic form $Q_{\Gamma}$ which vanishes on $\R^{|w|}$ and on $H_1(\Gamma,\R)$ is equal to 
\begin{equation*}
Q_{\Gamma} \left( \sum_{e\in E(\Gamma)} \alpha_{e}\cdot e  \right) = \sum_{e\in E(\Gamma)} \alpha_e^2\cdot \ell(e). 
\end{equation*}
See \cite[Definition~5.1.1]{BrannettiMeloViviani} for details.

\subsection{2-isomorphism of tropical curves}  The \textit{cycle matroid} of a weighted graph $\wG$ is the cycle matroid of $G$ with the addition of $|w|$ loops. Two weighted graphs are \textit{2-isomorphic} if there is a  bijection on the edge sets that induces an isomorphism of their cycle matroids. The $2$-isomorphism class of $\wG$ is written as $[\wG]_2$. A theorem of Whitney \cite{Whitney} asserts that two (unweighted) graphs are 2-isomorphic if and only if they are connected by a sequence of vertex gluings and twists about separating pairs of vertices. In particular,  $[G]_2$ consists of only $G$ whenever $G$ is 3-vertex connected. We refer the reader to \cite[Chapter~5.3]{Oxley} for details. 

Two tropical curves $\Gamma = (\wG,\ell)$ and $\Gamma' = (\wG', \ell')$ of the same genus are \textit{2-isomorphic} if there is a length-preserving bijection on the edge sets that induces an isomorphism on the cycle matroids of $\wG$ and $\wG'$. The $2$-isomorphism class of $\Gamma$ is written as $[\Gamma]_2$.

\subsection{Connectivity and C1-sets}

Let $\wG=(G,w)$ be a weighted graph. Then $\wG$ is \textit{2-connected} if $w(v) = 0$ for all $v$ and $G$ has no cut-vertices.   Every weighted graph admits a decomposition into \textit{blocks}, i.e., subgraphs that are either a single vertex of weight $1$ or a maximal 2-connected subgraphs of $\wG$.  See \Cref{fig:blockDecomposition}  for an example of such a decomposition.

\begin{figure}[htb]
	\begin{center}
			\includegraphics[height=25mm]{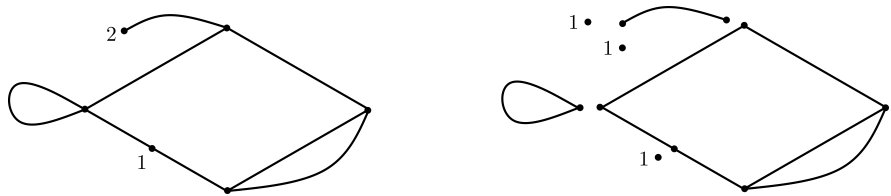}
	\end{center}
	\caption{A decomposition of a weighted graph into blocks. \label{fig:blockDecomposition}}
\end{figure}

Denote the set of nonseparating edges of $\wG$ by $E(\wG)_{\ns}$. For $e,f\in E(\wG)_{\ns}$, we say that $e\sim f$ if $e=f$ or $(e,f)$ form a separating pair of edges. This determines an equivalence relation on $E(\wG)_{\ns}$ (see \cite[Lemma~2.3.2]{CaporasoViviani}) whose equivalence classes are called \textit{C1-sets}.  Write $\Sets^1(\wG)$ for the collection of C1-sets. The C1-set that contains $f\in E(\wG)_{\ns}$ is denoted by $S_{f}$. 

The weighted graph $\wG$ is \textit{$k$-edge connected} if $G$ has at least $2$ edges, and the graph obtained by removing any $k-1$ edges from $G$ is connected. The \textit{2-edge connectivization} of $\wG$, written $\wG^2$, is obtained by contracting all separating edges of $\wG$. Consider the following operation on $\wG$.

\begin{itemize}
	\item[(C)] Given $S'\subset S$ for $S \in \Sets^1(\Gamma)$ and $e_0\in S'$, contract all edges in $S'$ except $e_0$. 
\end{itemize}

\noindent A \textit{3-edge connectivization} of $\wG$ is a weighted graph $\wG^3$ formed by forming $\wG^2$, then applying move (C) to all C1-sets of $\wG^2$. Two weighted graphs are \textit{C1-equivalent} if they belong to the same equivalence class of the equivalence relation generated by (C). By \cite[Lemma~2.3.8]{CaporasoViviani}, any two 3-edge connectivizations of a weighted graph are 2-isomorphic, and the contraction map $\wG \to \wG^3$ induces a bijection between $\Sets^1(\wG)$ and $E(\wG^3)$.  Given $S$ in $\Sets^1(\wG)$, let $e_S$ denote the edge under this correspondence. Then we have map $\psi:E(\wG)_{\ns} \to E(\wG^3)$ given by sending $f$ to $e_{S_f}$.

 The C1-sets of a tropical curve $\Gamma = (\wG,\ell)$ are the C1-sets of $\wG$ and the \textit{2-edge connectivization} $\Gamma^2$ is obtained by contracting all separating edges of $\Gamma$. A 3-edge connectivization $\Gamma^3$  of $\Gamma$ is formed in a manner similar to that of a weighted graph, but the edge lengths are modified so that $\Sets^1(\Gamma) \to E(\Gamma^3)$ is volume preserving. More precisely, consider the following operation on $\Gamma$. 

\begin{itemize}
	\item[(C')] Given $S'\subset S$ for $S$ in $\Sets^1(\Gamma)$ and $e_0\in S'$, contract all edges in $S'$ except $e_0$  and set the length of $e_0$  to $\sum_{e\in S'} \ell(e)$. 
\end{itemize}

\noindent A \textit{3-edge connectivization} of $\Gamma$ is a tropical curve $\Gamma^3 = (\wG^3,\ell^3)$ obtained first forming $\Gamma^2$ and then by applying move (C') to all C1-sets of $\Gamma^2$. Any two 3-edge connectivizations of a tropical curve are 2-isomorphic, see \cite[Remark~4.1.8]{CaporasoViviani}.  We say that $\Gamma$ and $\Gamma'$ are \textit{C1-equivalent}, written $\Gamma \sim_{C1} \Gamma'$, if $\Gamma^2$ and $\Gamma'^2$ belong to the same equivalence class of the equivalence relation generated by move (C').

\begin{proposition} 
	\label{prop:connectivizationProperties}
	Let $\Gamma = (\wG,\ell)$ be a tropical curve and $\wG'$ a weighted graph. 
	\begin{enumerate}
		\item If $\Gamma \sim_{C1} \Gamma'$  then $[\Gamma^3]_2 = [\Gamma'^3]_2$.
		\item If $\wG \sim_{C1} \wG'$, then there exists a  $\ell'$ such that $(\wG, \ell) \sim_{C1} (\wG',\ell')$.
		\item If $\epsilon: \wG'^3 \to \wG^3$ is a 2-isomorphism, then there is a bijection $\beta:\Sets^1(\wG') \to \Sets^1(\wG)$ so that $\beta(S_{e'}) = S_{\epsilon(e')}$. 
	\end{enumerate}
\end{proposition}

\begin{proof}
	Statements (1) and (3) are clear.  For (2), it suffices to consider the case when $\wG'$ is obtained from $\wG$ by applying (C) to $(S,e_0)$ where $S$ is a C1-set and $e_0\in S$. Define $\ell'(e_0) = \sum_{e\in S} \ell(e)$, and $\ell'(e) = \ell(e)$ for $e\in E(\wG)\setminus S$. Applying move (C') to $(\wG,\ell)$ yields $(\wG',\ell')$, as required.  
\end{proof}

\subsection{Hyperelliptic tropical curves} For a more comprehensive treatment of hyperelliptic tropical curves, we refer the reader to  \cite{Chan} in the unweighted case, or \cite[Section~4.11]{AminiBakerBrugalleRabinoff} in the weighted case. 

Let $\Gamma = (\wG,\ell)$ be a  tropical curve. An \textit{involution} of $\Gamma$ is an involution of the underlying graph of $\wG$ that preserves the weight and length functions. Note that swapping the two half edges of a loop is a nontrivial involution. If $\tau$ exchanges the two half edges $e$, then $e$ is said to be \textit{flipped}. The quotient of $\Gamma$ by $\tau$ is the (unweighted) tropical curve $\Gamma/\tau$ whose vertices of $\Gamma/\tau$ correspond to orbits of the action of $\langle\tau\rangle \subset \Aut(\wG)$ on $V(\wG)$, and the edges of $\Gamma/\tau$ correspond to the orbits of the non-flipped edges of $\wG$ The length of $[e] \in E(\Gamma/\tau)$ is $|\Stab(e)| \cdot \ell(e)$. In particular, a flipped edge is collapsed to a vertex upon forming $\Gamma/\tau$.  

A  tropical curve of genus at least $2$ is \textit{hyperelliptic} if its stable model $\Gamma$ has an involution $\tau$ such that each vertex of positive weight is fixed, and the underlying graph of $\Gamma/\tau$ is a tree.  If  $\Gamma$ is hyperelliptic, then there exists a unique $\tau$ that fixes all separating edges pointwise. This $\tau$ is called the \textit{hyperelliptic involution} of $\Gamma$.

We end this section with a discussion of fixed points and C1-sets of a stable hyperelliptic tropical curve. Let $\Gamma$ be a stable hyperelliptic tropical curve, $\tau$ its hyperelliptic involution,  $T = \Gamma/\tau$, and $\pi:\Gamma \to T$ the quotient map. A \textit{fixed point} of $\tau$ is a vertex $v$ of some subdivision of $\Gamma$ so that $\tau(v) = v$. 

\begin{lemma}
	\label{lem:3flippedEdges}
If $a$ is a 1-valent vertex of $T$ so that the vertices in  $\pi^{-1}(a)$ have weight 0,
then $\pi^{-1}(a)$ contains at least 2 edges. In particular, $\pi^{-1}(a)$ contains at least 2 fixed points, each appearing as the midpoint of a flipped edge.  
\end{lemma}

\begin{proof}
	Suppose $\pi^{-1}(a)$ contains only one weight 0 vertex $v$. Then $\pi$ does not contract any edges to $a$, so the valence of $v$ is at most 2, contradicting the stability assumption. Therefore $\pi^{-1}(a)$ contains another vertex $v'$.  There are at least 2 edges between $v$ and $v'$, otherwise these vertices would be 1 or 2-valent. 
\end{proof}

\begin{proposition}
	\label{prop:C1SetsHyperelliptic}
	Let $S$ be a C1-set of $\Gamma$. Then either
	\begin{itemize}
		\item $S$ has one edge and $\tau$ flips it, or
		\item $S$ has two edges and $\tau$ exchanges them. 
	\end{itemize}
\end{proposition}

\begin{proof}
Without loss of generality, assume $\Gamma$ is 2-edge connected.  Let $e=uu'$ and $f=vv'$ be distinct edges.  It suffices to show that $(e,f)$ form a separating pair if and only if $\tau(e) = f$. 

\textit{Case 1}. Suppose $e$ and $f$ are both flipped. By symmetry, it suffices to find a path from $v$ to $v'$ avoiding $e$ and $f$. Let $w$ be any fixed point of $\tau$ not contained in $e$ or $f$; such a fixed point exists by \Cref{lem:3flippedEdges}. There are paths from $v$ and $v'$ to $w$ not passing through any flipped edge (except possibly one that contains $w$).

\textit{Case 2}. Suppose $f$ is flipped but $e$ is not. Let $w$ be a fixed point lying above a 1-valent vertex of $T$  in the component of $T \setminus \pi(e)$ containing $\pi(f)$. We may choose $w$ so that it is not in $f$.  There are paths from $v$ and $v'=\tau(v)$ to  $w$ passing through only non-flipped edges (expect possibly the one containing $w$) avoiding $e$. This produces a path connecting $v$ and $v'$.

Now we construct a path from $u$ to $u'$ in $\Gamma\setminus\{e,f\}$. In a manner similar to the previous paragraph, we can find paths from $u$ to $\tau(u)$ and $u'$ to $\tau(u')$ avoiding $e$ and $f$. Together with $\tau(e)$, this gives the requisite path. 

\textit{Case 3}. Suppose neither $e$ or $f$ are flipped. If $\tau(e) = f$, then $(e,f)$ is a separating pair of edges. Now assume that $\tau(e) \neq f$. By symmetry, it suffices to construct a path from $v$ to $v'$ avoiding $e$ and $f$. If there are paths in $T$ from $\pi(v)$ and $\pi(v')$ to a 1-valent vertex that avoid $\pi(e)$ and $\pi(f)$, then connecting $v$ and $v'$ is similar to Case 2. Otherwise, every vertex $a$ in the maximal subgraph between $\pi(e)$ and $\pi(f)$ is 2-valent. By stability, each $\pi^{-1}(a)$ contains a fixed point, which can be used to construct the desired path from $v$ to $v'$.
\end{proof}

\section{Hyperelliptic type and its properties}
\label{sec:HyperellipticTypeAndItsProperties}

A tropical curve $\Gamma$ is said to be of \textit{hyperelliptic type} if there is a hyperelliptic tropical curve $\Gamma'$ such that $(\Jac(\Gamma), Q_{\Gamma}) \cong (\Jac(\Gamma'), Q_{\Gamma'})$. Such a $\Gamma'$ is called a \textit{hyperelliptic model} of $\Gamma$. Since hyperelliptic type is preserved under tropical equivalence, we are free to assume that the underlying weighted graph of our a hyperelliptic type tropical curve is stable. As a consequence of the tropical Torelli theorem~\cite[Theorem~4.1.9]{CaporasoViviani} (and~\cite[Theorem~5.5.3]{BrannettiMeloViviani} in the vertex-weighted case), we have the following characterization of hyperelliptic type tropical curves.

\begin{proposition}
	\label{prop:torelliHET}
	A genus $g\geq 2$ tropical curve $\Gamma$ is of hyperelliptic type if and only if there is a hyperelliptic tropical curve $\Gamma'$ such that $\cyc{\Gamma^3} = \cyc{\Gamma'^3}$. 
\end{proposition}

\noindent We say that  $\Gamma$ is  \textit{strongly} of hyperelliptic type if there is a choice of edge lengths that make it hyperelliptic. By~\Cref{prop:HETLengthIndependent} below, being hyperelliptic type does not depend on the edge lengths. Therefore strongly hyperelliptic type tropical curves are hyperelliptic type. However, the converse is not true. Consider the tropical curves in \Cref{fig:sameJacNotSHET}. The one on the left is hyperelliptic type (a hyperelliptic model is displayed on the right), but no choice of edge lengths will make it hyperelliptic. It is not even a specialization of a hyperelliptic tropical curve.

\begin{figure}[htb]
	\begin{center}
		\includegraphics[height=30mm]{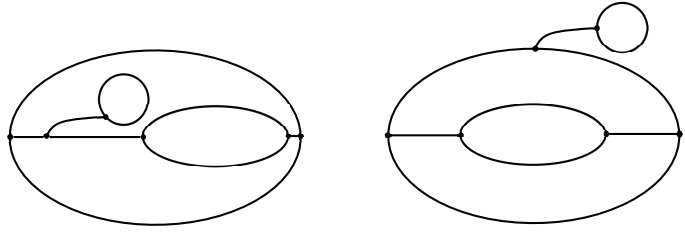}
	\end{center}
	\caption{On the left is a hyperelliptic tropical curve that is not strongly hyperelliptic type. On the right is a hyperelliptic model for this tropical curve. \label{fig:sameJacNotSHET}}
\end{figure}

Now we show that the property of being hyperelliptic type does not depend on the length function.

\begin{proposition}
	\label{prop:HETLengthIndependent}
	Suppose $\Gamma_1 = (\wG, \ell_1)$ is of hyperelliptic type,  and let $\Gamma_2 = (\wG, \ell_2)$ be a tropical curve with the same underlying weighted graph. Then $\Gamma_2$ is also of hyperelliptic type. 
\end{proposition}

\begin{proof}
	Without loss of generality, we may assume that $\wG$ is 2-edge connected. Suppose $\Gamma_1' = (\wG', \ell_1')$ is a stable hyperelliptic model for $\Gamma_1$.  Let $\beta:\Sets^1(\wG') \to \Sets^1(\wG)$ be the bijection of C1-sets as in~\Cref{prop:connectivizationProperties}(3). Define $\ell_2'$ on $\wG$ in the following way.  Given a C1-set $S$ of $\wG'$, 
	\begin{itemize}
		\item if $S=\{e_0\}$ set $\ell_2'(e_0) = \sum_{e\in\beta(S)} \ell_2(e)$, or
		\item if $S=\{e_0, f_0\}$ set $\ell_2'(e_0) = \ell_2'(f_0) = \frac{1}{2} \sum_{e\in\beta(S)} \ell_2(e)$.
	\end{itemize}
	Then $\Gamma_2' = (\wG',\ell_2')$ is hyperelliptic (see \Cref{prop:C1SetsHyperelliptic}) and  $\cyc{\Gamma_2^3} = \cyc{\Gamma_2'^3}$, so $\Gamma_2$ is of hyperelliptic type. 
\end{proof}

With this Proposition in mind, we say that a weighted graph $\wG$ is of \textit{hyperelliptic type} if $(\wG,\ell)$ is hyperelliptic type for some (and therefore, any) length function $\ell$. Similarly, we say that $\wG$ is \textit{strongly} of hyperelliptic type if $(\wG,\ell)$ is hyperelliptic for some $\ell$. 

\begin{proposition}
	\label{prop:SHETC1equiv}
If $\Gamma = (\wG,\ell)$ is strongly of hyperelliptic type, then it is C1-equivalent to a hyperelliptic tropical curve.
\end{proposition}

\begin{proof}
For each $e\in E(\Gamma)_{\ns}$, let $\ell'(e) = (\sum_{f\in S_e} \ell(f))/|S_e|$. If $e$ is a separating edge, set $\ell'(e) = \ell(e)$.  By \Cref{prop:C1SetsHyperelliptic}, $\Gamma' = (\wG,\ell')$ is a hyperelliptic tropical curve that is C1-equivalent to $\Gamma$.
\end{proof}

\noindent In \Cref{thm:HTK4L3} below, we will prove that a tropical curve of hyperelliptic type is C1-equivalent to a hyperelliptic tropical curve.

In order for the wedge sum of two hyperelliptic tropical curves to be hyperelliptic, they need to be attached at fixed points of the respective hyperelliptic involutions. As we will see in the next two Lemmas, such a wedge sum is of hyperelliptic type regardless of how we decide to glue. 
Given tropical curves $\Gamma_1 = (\wG_1,\ell_1)$, $\Gamma_2 = (\wG_2, \ell_2)$ and vertices $v_1$, $v_2$ of some subdivision of $\wG_1$, $\wG_2$ respectively, let $\Gamma_1\vee_{v_1,v_2} \Gamma_2$ denote the tropical curve obtained by gluing $\Gamma_1$ and $\Gamma_2$ at $v_1$ and $v_2$. 

\begin{lemma}
	Let $\Gamma_1$ and $\Gamma_2$ be weighted metric graphs, and $v_i, v_i'$ vertices of some subdivision of $\Gamma_i$. Set $\Gamma = \Gamma_1 \vee_{v_1, v_2} \Gamma_2$ and $\Gamma' = \Gamma_1 \vee_{v_1', v_2'} \Gamma_2$. Then $\Gamma \sim_{C1} \Gamma'$. In particular, any tropical curve $\Gamma$ is C1-equivalent to an arbitrary wedge sum of its blocks. 
\end{lemma}

\begin{proof}
	It suffices to consider the case where $v_2=v_2'$ and $v_1$ is connected to $v_1'$ by an edge $e$. Subdivide $e$ into 2 edges, and let $w$ be the new vertex. Then $\Gamma$ and $\Gamma'$ are C1-equivalent to $\Gamma_1\vee_{w,v_2}\Gamma_2$. 
\end{proof}

\begin{lemma} 
	\label{lem:HET2ConComponents}
	Suppose that each block of $\Gamma^2$ of genus $\geq 2$ is C1-equivalent to a strongly hyperelliptic type tropical curve. Then $\Gamma$ is C1-equivalent to a hyperelliptic tropical curve. In particular $\Gamma$ is of hyperelliptic type.  
\end{lemma}

\begin{proof}
	Let $\{\Gamma_i\}_{i=1}^k$ be the blocks of $\Gamma^2$, and let $G_i$ be the underlying (unweighted) graph of $\Gamma_i$. If $g(G_i) \geq 2$, let $\Gamma_i'$ be a hyperelliptic tropical curve that is C1-equivalent to $\Gamma_i$, and $v_i$ a fixed point for its hyperelliptic involution. Note that $\Gamma_i'$ exists by \Cref{prop:SHETC1equiv}, and fixed points exist by \Cref{lem:3flippedEdges}. If $g(G_i) = 1$ let $\Gamma_i'$ be the graph with a single vertex $v_i$ and a loop whose length is the sum of edge lengths of $\Gamma_i$. Finally, if $g(G_i) = 0$, then $\Gamma_i$ consists of a single vertex $v_i$ of weight $1$. In this case set $\Gamma_i' = \Gamma_i$. Then $\vee_{v_i} \Gamma_i'$ is a hyperelliptic tropical curve that is C1-equivalent to $\Gamma$.   
\end{proof}

If $e\in E(\Gamma)_{\ns}$, then $\Gamma^3\setminus \psi(e)$ may not be a 3-edge connectivization of $\Gamma\setminus e$. However, both of these tropical curves will have a common 3-edge connectivization, as we will see in the following Lemma.

\begin{lemma}
	\label{lem:3edgeConnectivizationEdgeRemoval}
	Given a tropical curve $\Gamma$ and $e \in E(\Gamma)_{\ns}$, $(\Gamma^3 \setminus \psi(e))^3$ is a 3-edge connectivization of $\Gamma\setminus e$. 
\end{lemma}

\begin{proof} 
First, note that $\psi:E(\Gamma)_{\ns} \to E(\Gamma^3)$ induces a map $\gamma : \Sets^1(\Gamma\setminus e) \to \Sets^1(\Gamma^3\setminus \psi(e))$ by $	\gamma(S) = \{ \psi(f) \, | \, f\in S \}$. Let $S$ be a C1-set of $\Gamma \setminus e$ and write
\begin{align*}
\gamma(S) = \{g_1,\ldots,g_{k}\}, && S = \{f_{ij} \, | \, 1\leq i \leq k, 1 \leq j \leq |\psi^{-1}(g_i)|\},
\end{align*}
where $\psi(f_{ij}) = g_i$. The Lemma now follows from the fact that applying move (C') to $(S,f_{11})$ is the same as applying (C') to each $(\psi^{-1}(g_i),f_{i1})$, then to $(\gamma(S), g_1)$.  
\end{proof}

Now we will show that the property of being hyperelliptic type is a minor closed condition on stable weighted graphs. A \textit{minor} of $\wG$ is a weighted graph $\wG'$ obtained by a sequence of lowering weights, removing edges, or performing weighted contractions. 

\begin{proposition}
	\label{prop:HETMinorClosed}
	Suppose $\wG$ is a weighted stable graph of hyperelliptic type, and $\wG'$ is a genus $g\geq 2$ connected minor. Then $\wG'$ is also of hyperelliptic type. 
\end{proposition}

\begin{proof}
	First, consider the case where $\wG$ is strongly of hyperelliptic type, say $\Gamma = (\wG,\ell)$ is hyperelliptic and $\tau$ its involution. If $\Gamma'$ is obtained by deducting a weight by 1, then $\tau$ is still a hyperelliptic involution for $\Gamma'$.  If $e \in E(\wG)_{\ns}$ is not in a separating pair, then $\tau$ flips it by \Cref{prop:C1SetsHyperelliptic}. This means that $\Gamma/e$ and $\Gamma\setminus e$ remain hyperelliptic. 
	
	Now suppose $e$ is in a pair of separating edges $(e,f)$.  Define a new length function $\ell'$ that agrees with $\ell$ except $\ell'(e) = \ell'(f) = \ell(f)/2$. Then $(\wG,\ell')$ is a hyperelliptic model of $\Gamma/e$. Finally, consider $\Gamma' = \Gamma \setminus e$. Then $f$ is a separating edge of $\Gamma'$, say  $\Gamma_1$ and $\Gamma_2$ are the components of $\Gamma' \setminus f$. The map $\tau$ restricts an involution on both $\Gamma_1$ and $\Gamma_2$ such that $\Gamma_i/\tau$ is a tree. By~\Cref{lem:HET2ConComponents}, $\Gamma'$ is of hyperelliptic type. 
	
	For the general case, fix a length function $\ell$ for $\wG$ and let $\Gamma' = (\wG',\ell')$ is a hyperelliptic model for $\Gamma = (\wG,\ell)$. As 3-edge connectivization and 2-isomorphism are independent of the weights, deducting a weight from $\Gamma$ produces a tropical curve that is of hyperelliptic type (as long as the genus remains at least $2$).
	Fix a nonseparating edge $e$ of $\wG$. Suppose $\epsilon: E(\wG^3) \to E(\wG'^3)$ is a 2-isomorphism, and let  $\psi:E(\Gamma)_{\ns} \to E(\Gamma^3)$, $\psi':E(\Gamma')_{\ns} \to E(\Gamma'^3)$ be 3-edge connectivizations. Choose $e'\in E(\Gamma')_{\ns}$ such that $\psi(\epsilon(e)) = \psi'(e')$. Then 
	\begin{align*}
	\cyc{\Gamma^3\setminus \psi(e)} = \cyc{\Gamma'^3\setminus \psi'(e')} && \cyc{\Gamma^3 / \psi(e)} = \cyc{\Gamma'^3 / \psi'(e')}
	\end{align*}
	and therefore
	\begin{align*}
	\cyc{(\Gamma^3\setminus \psi(e))^3} = \cyc{(\Gamma'^3\setminus \psi'(e'))^3} && \cyc{(\Gamma^3 / \psi(e))^3} = \cyc{(\Gamma'^3 / \psi'(e'))^3}.
	\end{align*}
	By applying~\Cref{lem:3edgeConnectivizationEdgeRemoval} in the edge removal case, we see that
	\begin{align*}
	\cyc{(\Gamma\setminus e)^3} = \cyc{(\Gamma'\setminus e')^3} && \cyc{(\Gamma/e)^3} = \cyc{(\Gamma'/ e')^3}.
	\end{align*}
	The Proposition now follows from the strongly hyperelliptic type case. 
\end{proof}

To a stable weighted graph $\wG$, let $d(\wG)$ be
\begin{equation}
\label{eqn:dG}
d(\wG) = \sum_{v\in V(\wG)} \val(v) + 3w(v) - 3.
\end{equation}
Note that $d(\wG)$ is the dimension of the stratum of $\overline{\cM}_g$ of stable curves whose weighted dual graph is $\wG$. Note that $d(\wG) \geq 0$ with equality if and only if $w(v) = 0$ for all $v$ and $G$ is trivalent. For a tropical curve $\Gamma$, let $d(\Gamma) = d(\wG)$ where $(\wG,\ell)$ is a stable tropical curve tropically equivalent to $\Gamma$.  If $\wG'$ is obtained from $\wG$ by contracting an edge, then $d(\wG') = d(\wG) + 1$. In particular, if $d(\wG^3) = d(\wG)$ then $\wG$ is 3-edge connected.

\begin{proposition}
	\label{prop:K4L3NotHT}
	The graphs $K_4$ and $L_3$ are not of hyperelliptic type.
\end{proposition}

\begin{proof}	
	First consider the $K_4$ case. Suppose $K_4$ is of hyperelliptic type and let $\Gamma = (\wG, \ell)$ be a stable hyperelliptic model for $K_4$. By the above comments, we see that
	\begin{equation*}
	0 \leq d(\Gamma) \leq d(\Gamma^3) = d(K_4) = 0,
	\end{equation*}
	so $\wG$ is already 3-edge connected. This means that $\cyc{\wG} = \cyc{K_4}$, and therefore $\wG = K_4$ since $K_4$ is 3-vertex connected. Because $K_4$ has no separating pairs of edges, the hyperelliptic involution of $\Gamma$ flips each edge of $K_4$ as in \Cref{prop:C1SetsHyperelliptic}.  This is a contradiction since no automorphism of $K_4$ satisfies this property. 
	
	Now suppose that $L_3$ is of hyperelliptic type, and that $\Gamma=(\wG,\ell)$ is a stable hyperelliptic model. Any stable graph differing from $L_3$ by a single edge contraction is 3-edge connected, so $\cyc{\wG} = \cyc{L_3}$.  Because $L_3$ has no cut vertices or separating pairs of vertices, $\cyc{L_3}$ consists of just $L_3$, so $\wG = L_3$.  As in the $K_4$ case, this means that the hyperelliptic involution of $\Gamma$ flips each edge of $L_3$, which is a contradiction.
\end{proof}

\noindent Observe that the ``only if'' direction of the theorem in the Introduction follows from~\Cref{prop:HETMinorClosed} and~\Cref{prop:K4L3NotHT}.

\section{Nested ear decompositions}
\label{sec:NestedEarDecompositions}
Nested ear decompositions were introduced in \cite{Eppstein} to study series-parallel graphs. Since these graphs are characterized by the absence of a $K_4$ minor, it is natural to explore the link between nested ear decompositions and tropical curves of hyperelliptic type. In this section, we will introduce the notion of a hyperelliptic type adapted (nested) ear decomposition, and show that any graph with no $K_4$ or $L_3$ minor admits such a decomposition. This will then be used to prove the ``if'' direction of the theorem in the Introduction. 

Let $G$ be a finite connected graph. An \textit{ear decomposition} of $G$ is  a collection of paths   $\ED = \{E_0, \ldots, E_g\}$ called \textit{ears} that partition $E(G)$ and satisfying the following properties.
\begin{enumerate}
	\item If two vertices in an ear are the same, then they must be the two endpoints of the same ear.
	\item The two endpoints of $E_k$ ($k\geq 1$) appear in $E_i$ and $E_j$ for $i,j < k$.
	\item No interior vertex of $E_j$ is in $E_i$ for $i\leq j$.
\end{enumerate}
An ear decomposition $\ED = \{E_0,\ldots, E_g\}$ is \textit{open} if the two endpoints of each ear are distinct. The ear $E_j$ is \textit{nested} in $E_i$ if $i<j$ and the endpoints of $E_j$ are in $E_i$. In this case, the \textit{nest interval} of $E_j$ in $E_i$ is the path $E_{ij}$ in $E_i$ between the endpoints of $E_j$. We write $E_{ij}^{\circ}$ for the interior of  $E_{ij}$. An ear decomposition $\ED$ is \textit{nested} if it is open and satisfies the following. 
\begin{enumerate}
	\item For $1\leq j \leq g$ there is some $i<j$ such that $E_j$ is nested in $E_i$.
	\item If $E_j$ and $E_k$ are both nested in $E_i$, then either $E_{ij}$ and $E_{ik}$ have disjoint interiors, or one is contained in the other.  
\end{enumerate}

Let $G$ be a graph with a nested ear decomposition $\ED$. The ear $E_j$ is \textit{properly nested} in $E_i$ if $E_j$ has an endpoint in the interior of $E_i$ and the other endpoint does not lie in the interior of any ear $E_k$ for $k>i$. If no such ear exists, i.e. the endpoints of $E_j$ coincide with those of $E_0$, then $E_j$ is called an \textit{initial ear}.    By \cite[Lemma~3]{Eppstein}, every non-initial ear is properly nested in exactly one other ear. If $E_j$ is properly nested in $E_i$, we write $E_i \lessdot E_j$.  Taking the reflexive and transitive closure of $\lessdot$ induces a partial order $\leq$ on $\ED$. 

\begin{proposition}
	\label{prop:K4NestedEarDecomposition}
	A 2-connected graph $G$ admits a nested ear decomposition $\ED$  if and only if $G$ has no $K_4$ minor.  
\end{proposition}

\begin{proof}
	By \cite[Theorems~1,~3]{Duffin}, $G$ has no $K_4$ minor if and only if it is series-parallel. By \cite[Theorem~1]{Eppstein}, $G$ is series-parallel if and only if $G$ has a nested ear decomposition. 
\end{proof}

A \textit{hyperelliptic type adapted ear decomposition} (HTED) is a nested ear decomposition $\ED$ that satisfies the following additional property: if $E_{j}$ and $E_k$ are properly nested in $E_{i}$, then $E_{ij} \subset E_{ik}$ or $E_{ik} \subset E_{ij}$. A \textit{hyperelliptic adapted ear decomposition} (HED) is a HTED satisfying the following.
\begin{enumerate}
	\item  If $E_i \lessdot E_j$, then the endpoints of $E_j$ lie in the interior of $E_i$.	
	\item If $E_{j}$, $E_{k}$ are nested in $E_i$ and  $E_{ij} \subset E_{ik}$ then $E_{ij} = E_{ik}$ or the endpoints of $E_j$ lie in $E_{ik}^{\circ}$.
\end{enumerate}

\begin{lemma}
	\label{lem:existHTED}
	If $G$ is a 2-connected graph of genus $\geq 2$ that has no $K_4$ or $L_3$ minor, then $G$ has a HTED. 
\end{lemma}

\begin{proof}	
	By~\Cref{prop:K4NestedEarDecomposition}, $G$ has a nested ear decomposition $\ED = \{E_0,\ldots,E_g\}$. By 2-connectedness, there are at least two initial ears $E_0$ and $E_1$. Label their endpoints by $s$ and $t$. Suppose $\ED$ is not hyperelliptic type adapted, i.e., there is an ear $E_i$ with at least 2 ears $E_j$ and $E_k$ properly nested in it whose nest intervals have disjoint interiors. Assume  these are chosen maximally in the sense  that if  $E_i \lessdot E_{\ell}$  then $E_{i\ell}$ does not contain exactly one of  $E_{ij}$ or $E_{ik}$.  Without loss of generality, suppose that $E_1 \leq E_i$.  We claim that $\ED$ consists of the following. 
	
	\begin{enumerate}
		\item It has exactly 2 initial ears $E_0$ and $E_1$.
		\item The only ear nested in $E_0$ is $E_1$. 
		\item The ears $E_j$ and $E_k$ are nested in $E_1$, and their nest intervals have disjoint interiors (take $E_j$ to be the ear closer to $s$ along $E_1$). 
		\item If $E_m$ is nested in $E_1$, then $E_{1m}$ is contained in either $E_{1j}$ or $E_{1k}$. 
		\item If $E_{m}$ and $E_n$ are nested in $E_1$ with $E_{1m}$ and $E_{1n}$  contained in $E_{1j}$ (resp. $E_{1k}$), then $E_{1m}$ contains $E_{1n}$ or vice versa. 
		\item For every $E_{\ell} \neq E_{0}, E_{1}$, if $E_{m}$ and $E_{n}$ are nested in $E_{\ell}$, then $E_{\ell m}$ contains $E_{\ell n}$ or vice versa. 
	\end{enumerate} 
If $E_{\ell} \neq E_1$  is nested in $E_0$ (possibly an initial ear), then any connected subgraph that contains $E_0$, $E_i$, $E_j$, $E_k$,  and $E_{\ell}$ has a $L_3$ minor. Therefore the only ear nested in $E_0$ is $E_1$. This proves (1) and (2).  If $E_{i} \neq E_1$, then any connected subgraph containing $E_0$, $E_1$, $E_i$, $E_j$, and $E_k$ has a $L_3$ minor. Therefore $E_i = E_1$, demonstrating (3).
	
Suppose $E_{m}$ is nested in $E_{1}$. If $E_{1m}$ contains $E_{1j}$ and $E_{1k}$  then $E_{0} \cup E_{1} \cup E_j \cup E_k \cup E_m$ has a $L_3$ minor. The same is true if $E_{1m}^{\circ}$ is disjoint from $E_{1j}^{\circ}$ and $E_{1k}^{\circ}$. Together with the maximality assumption on $E_j$ and $E_k$, this proves (4).  	Now suppose $E_m$ and $E_n$ are nested in $E_{1}$ with $E_{1m}, E_{1n} \subset E_{1j}$. If $E_{1m}^{\circ}$ and $E_{1n}^{\circ}$ are disjoint, then $E_{0} \cup E_{1} \cup E_{k} \cup E_{m} \cup E_{n}$ contains a $L_3$ minor. Similarly, if $E_{1m}, E_{1n} \subset E_{1k}$ then they cannot have disjoint interiors, proving (5). For (6), suppose $E_{\ell} \neq E_{0}, E_1$ and that $E_m$, $E_n$ are nested in it. If $E_{\ell m}^{\circ}$ and $E_{\ell n}^{\circ}$ are disjoint, then any connected subgraph containing  $E_{0}$, $E_{1}$, $E_{\ell}$, $E_{m}$, $E_n$ has a $L_3$ minor.  Thus $\ED$ has the required form. 
	
Define a new ear decomposition $\ED'$ in the following way. Write $t'$ for the endpoint of $E_k$ which is closer to $s$ along $E_1$. Let $E_0'$ be the path from $s$ to $t$ along $E_0$, followed by the path from $t$ to $t'$ along $E_1$. The next ear $E_1'$ will be the path from $s$ to $t'$ along $E_1$.  The remaining ears are left unchanged, i.e., $E_i'=E_i$ for $i\geq 2$. Then $\ED'$ is a HTED.  For an illustration of this modification, see~\Cref{fig:ModifyToHTED}. 
\end{proof}

\begin{figure}[htb]
	\begin{center}
		\includegraphics[width=0.9\textwidth]{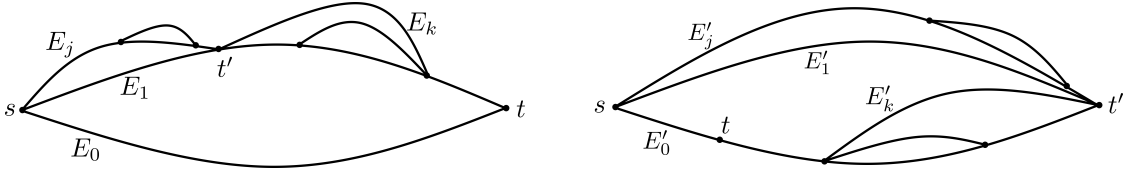}
	\end{center}
	\caption{Modification of a nested ear decomposition to a HTED on a graph with no $L_3$ minor. \label{fig:ModifyToHTED}}
\end{figure}

\begin{lemma}
	\label{lem:HTEDtoHED}
	Suppose $G$ is a 2-connected stable graph of genus $g\geq 2$ that has a HTED.
	\begin{enumerate}
		\item The graph $G$ has a HTED with at least $3$ initial ears.
		\item If $G$ is trivalent, then any HTED with 3 initial ears is a HED.
		\item There is a $G'$ that has a HED such that $G\sim_{C1} G'$ and $G$ is a specialization of $G'$.  
		
	\end{enumerate}
\end{lemma}

\begin{proof}
	Suppose $\ED = \{E_0, \ldots, E_g\}$ is a HTED for $G$ that only has two initial ears $E_0$ and $E_1$. Let $s,t$ be the endpoints of these ears. By stability, there is an ear $E_{j}$ ($i\geq 2$) that has $s$ as an endpoint, assume that it is nested in $E_0$.   Choose $E_j$ maximally in the sense that if $E_0 \lessdot E_k$ then $E_{0j}$ contains $E_{0k}$, and write $t'$ for the other endpoint of $E_j$  Define a new ear decomposition $\ED'$ as follows. Let $E_0'$ be the path from $s$ to $t'$ along $E_0$ and $E_1'$ the path from $s$ to $t$ along $E_1$, followed by the path from $t$ to $t'$ along $E_0$. Finally, let $E_i'=E_i$ for $i\geq 2$. Then $\ED'$ is a HTED such that $E_0', E_1', E_j'$ are initial ears.  This proves (1).
	
	Now suppose $G$ is trivalent and $\ED$ is a HTED with at least 3 initial ears. We claim that $\ED$ is already a HED. Suppose $E_j \lessdot E_k$. If $E_j$ is initial, then both endpoints of $E_k$ lie in the interior of $E_j$ since $E_k$ is not initial.  Now suppose $E_j$ is not initial, and that $a,b$ are the endpoints of $E_j$. Then $a$ lies in the interior of some ear $E_i$. Since $G$ is trivalent, $E_k$ cannot have $a$ as an endpoint. A similar argument shows that $b$ is not an endpoint of $E_k$.  Therefore the endpoints of $E_k$ lie in the interior of $E_j$, and so condition (1) of a HED is satisfied. 
	
	Next, assume that $E_{j}$ and $E_{k}$ are nested in $E_i$ with $E_{ij} \subset E_{ik}$. If $E_{k}$ is initial, then either $E_{j}$ is initial (in which case $E_{i}$, $E_{j}$, and $E_k$ are the initial ears) or the endpoints of $E_j$ lie in $E_{ik}^{\circ}$. Otherwise, the endpoints of $E_k$ lie in the interior of $E_{i}$. Similar to the case in the previous paragraph, the endpoints of $E_j$ lie in $E_{ik}^{\circ}$. 	 This verifies condition (2) of a HED and completes part (2) of this Lemma. 
	
	To prove (3), we proceed by induction on $d(G)$ defined in~\Cref{eqn:dG}. When $d(G) = 0$, $G$ is trivalent and $\ED$ is already a HED. Now suppose that the Lemma is true for stable graphs with $d<\delta$, and let $G$ be a stable graph with $d(G) = \delta$. If $\ED$ is not a HED, then there are ears satisfying at least one of the following:
		
		(a)   $E_i\lessdot E_j$, but one endpoint of $E_j$ coincides with an endpoint $a$ of $E_i$, or 
		
		 (b) $E_{ij} \subset E_{ik}$ but exactly one endpoint $b$ of $E_j$ lies in the interior of $E_{ik}$. 
		 
	\noindent Consider case (a). Assume $E_{j}$ is chosen so that if $E_i \lessdot E_{k}$ then $E_{ik} \subset E_{ij}$.  Let $G'$ be the graph obtained from $G$ in the following way. First subdivide the edge in $E_{ij}$ adjacent to $a$, creating a new vertex $b$. Let $\ED_a \subset \ED$ be the ears $E_{\ell}$ that have $a$ as an endpoint, and either $E_j\leq E_{\ell}$ or $E_{i\ell} \subset E_{ij}$. For every ear in $\ED_a$, move the corresponding endpoint from $a$ to $b$. See the left side of \Cref{fig:ModifyToHED} for an illustration. Let $e$ be the unique edge in $E_i \setminus E_{ij}$ and $f$ the edge between $a$ and $b$.  Then $(e,f)$ form a separating pair of edges for $G'$, and contracting $f$ yields $G$. Therefore $G'\sim_{C1}G$ and $d(G') = \delta - 1$. By the inductive hypothesis, there is a $G''$ that has a HED such that $G'' \sim_{C1} G'$ and $G$ is a specialization of $G''$. Case (b) is handled in a similar fashion, see the right side of \Cref{fig:ModifyToHED}.
\end{proof}

\begin{figure}[htb]
	\begin{center}
		\includegraphics[width=0.4\textwidth]{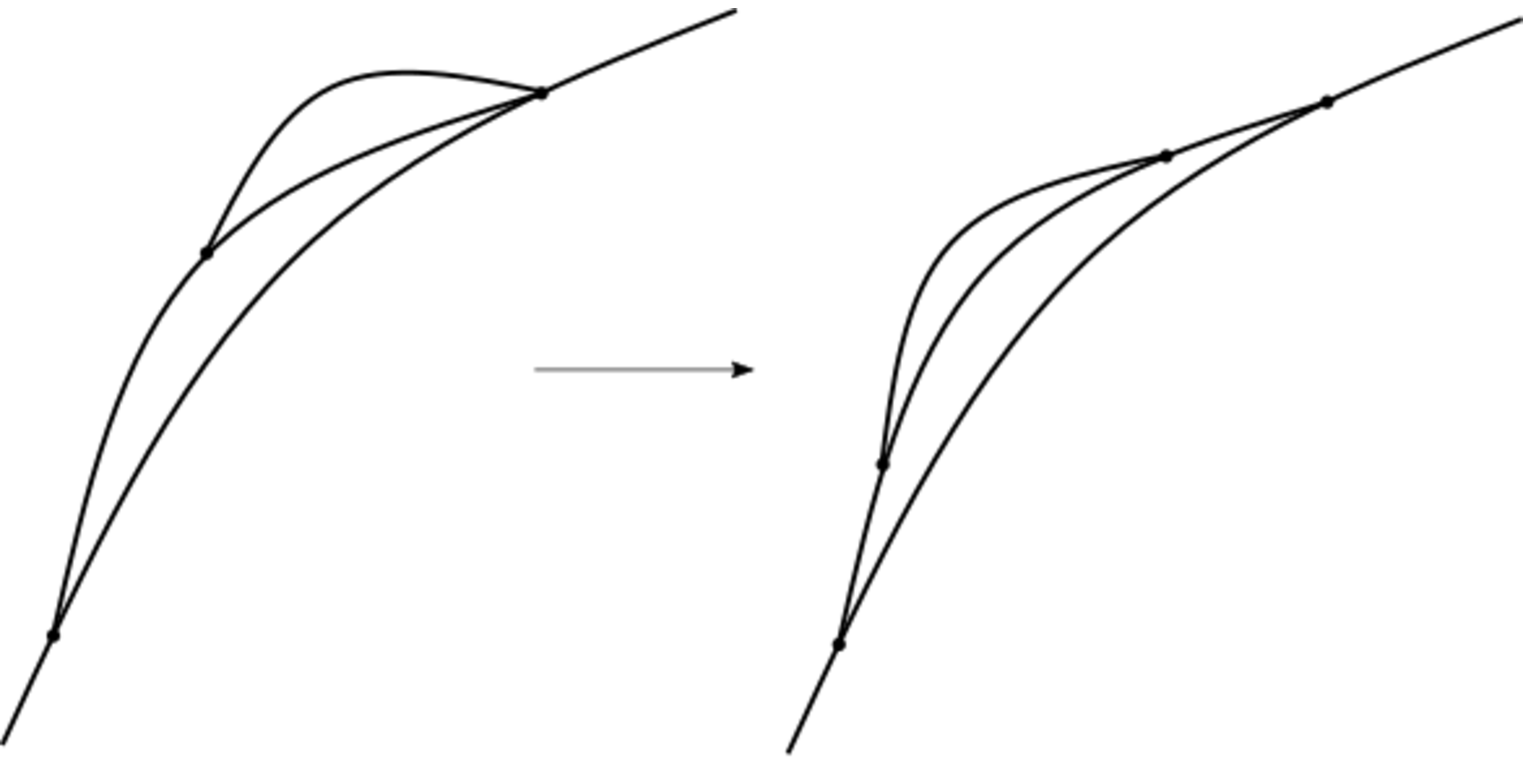} $\;\;\;\;\;\;\;$
		\includegraphics[width=0.4\textwidth]{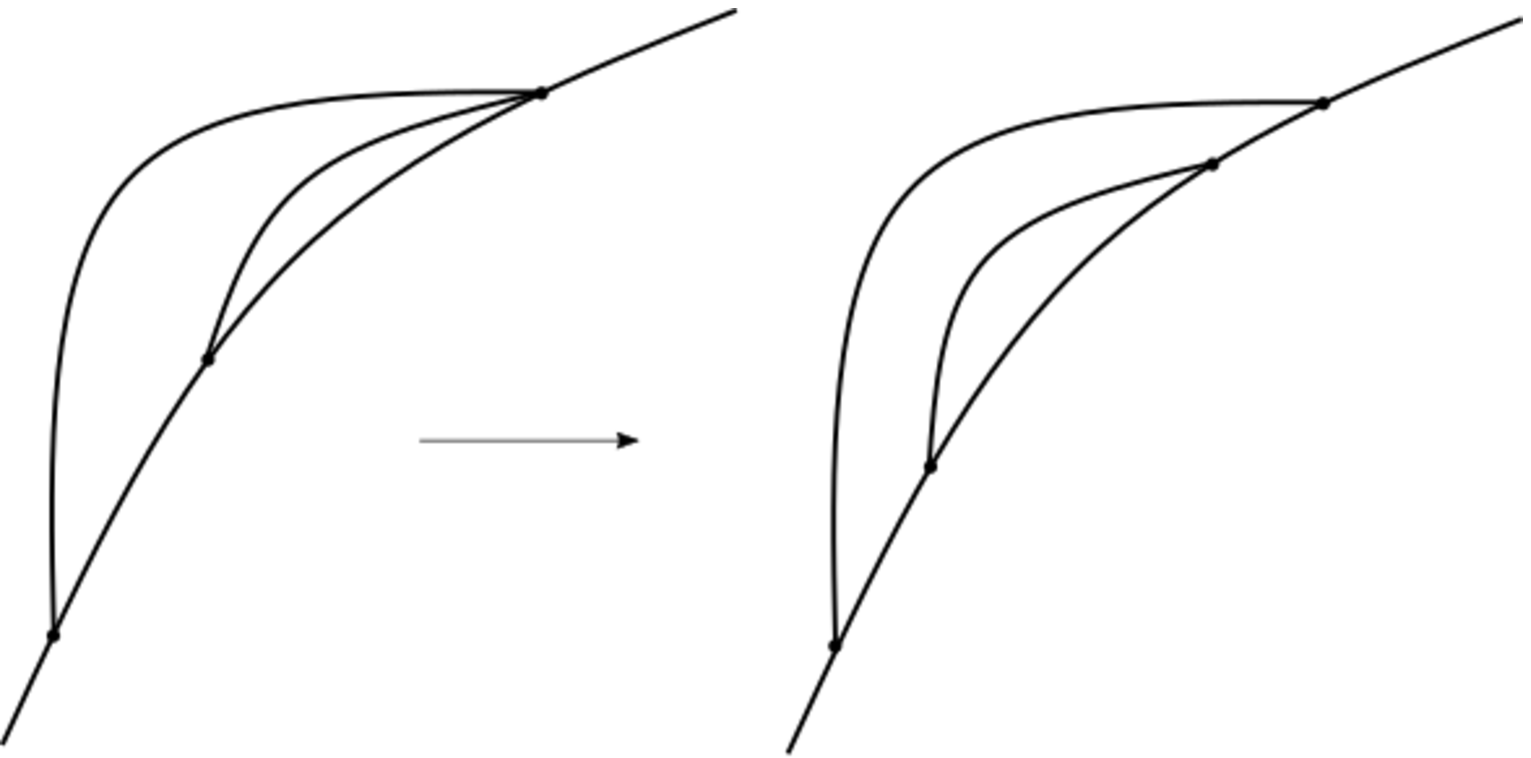}
	\end{center}
	\caption{Modifications in the inductive step of the proof of \Cref{lem:HTEDtoHED}(3). \label{fig:ModifyToHED}}
\end{figure}

\begin{lemma}
	\label{lem:HEDtoSHET}
	Suppose $G$ is a 2-connected stable graph of genus $g\geq 2$ that has a HED. Then $G$ is strongly hyperelliptic type.  
\end{lemma}

\begin{proof}
	Fix a HED $\ED = \{E_0,\ldots,E_g\}$ for $G$. We define $\tau$ ear by ear as follows.  Let $v, w$ be the endpoints for $E_i$ and set $\tau(v) = w$. If $E_i$ has no ears properly nested in it, then $E_i$ has exactly one edge. Let $\tau$ flip this edge. Otherwise, choose a maximal collection of ears $E_{j_1},\ldots, E_{j_k}$ properly nested in $E_i$ so that  $E_{ij_1} \subsetneq \cdots \subsetneq E_{ij_k} \subsetneq E_i$. 
	Now, $E_i\setminus E_{ij_1}$ separates the endpoints of each $E_{ij_a}$. Let $v_a$ be the endpoint of  $E_{ij_a}$ appearing on the side of $E_i \setminus E_{ij_1}$ that contains $v$, and $w_a$ the other endpoint. Let $e_{i,a} = v_av_{a+1}$, $e_{i,k} = v_kv$, $f_{i,a} = w_{a+1}w_a$, and $f_{i,k} = ww_k$. Define $\tau(e_{i,a}) = f_{i,a}$ and $\tau(f_{i,a}) = e_{i,a}$.  Finally, $E_{ij_1}$ consists of just one edge, so let $\tau$ flip it. This defines an involution $\tau$. Let $\pi: G \to G/\tau$ denote the quotient map.
	
	We claim that $G/\tau$ is a tree. Suppose $e$ is an edge of $G/\tau$. Then $\pi^{-1}(e)$ is an edge of $G$ that is not flipped. This means that $\pi^{-1}(e) = \{e_{i,a}, f_{i,a}\}$ for some $i$ and $a$, which is a separating pair of edges. So removal of $e$ from $G/\tau$ disconnects it. Since ever edge of $G/\tau$ is separating, it is a tree. Choosing $\ell$ so that $\ell(e_{i,a}) = \ell(f_{i,a})$ produces a hyperelliptic tropical curve $(G,\ell)$. 
\end{proof}

\begin{theorem}
	\label{thm:HTK4L3}
	Let $\Gamma = (\wG,\ell)$ be a tropical curve. The following are equivalent.
	\begin{enumerate}
		\item The tropical curve $\Gamma$ is of hyperelliptic type. 
		\item The underlying graph $G$ has no $ K_4 $ or $L_3$ minor. 
		\item There is a hyperelliptic $\Gamma'$ such that $\Gamma \sim_{C1} \Gamma'$.
	\end{enumerate}
	If in addition $\Gamma$ is 2-connected and unweighted, then the following is equivalent to the previous.
	\begin{enumerate}[resume]
		\item There is a hyperelliptic $\Gamma'$ with $\Gamma \sim_{C1} \Gamma'$ and $\Gamma'$ specializes to $\Gamma$.
	\end{enumerate}
\end{theorem}

\begin{proof}
	The implication (1) $\Rightarrow$ (2) follows from~\Cref{prop:HETMinorClosed} and~\Cref{prop:K4L3NotHT}.
	
	Now suppose $\Gamma$ is 2-connected and has no $K_4$ or $L_3$ minor. By~\Cref{lem:existHTED}, $G$ has a HTED.  By~\Cref{lem:HTEDtoHED}, there is a $G'$ that has a HED such that $G\sim_{C1}G'$ and $G'$ specializes to $G$. Moreover, $G'$ is strongly of hyperelliptic type by \Cref{lem:HEDtoSHET}.   By~\Cref{prop:connectivizationProperties}(2) and \Cref{prop:SHETC1equiv}, $\ell'$ may be chosen so that  $\Gamma \sim_{C1} \Gamma'$. This shows (2) $\Rightarrow$ (4), and (2) $\Rightarrow$ (3) now follows from \Cref{lem:HET2ConComponents}. 
	
	Finally, (3) $\Rightarrow$ (1) and (4) $\Rightarrow$ (1) follow from \Cref{prop:connectivizationProperties}(1) and \Cref{prop:torelliHET}. 
\end{proof}

\bibliographystyle{amsalpha}
\bibliography{HETbib}

\end{document}